\documentclass[11pt]{amsart}
\usepackage{amssymb,amsmath,mathrsfs}
\usepackage{enumerate,tikz,graphics}
\usepackage{hyperref}
\usepackage{xstring}

\usetikzlibrary{patterns}

\newtheorem{theorem}{Theorem}[section]
\newtheorem{corollary}[theorem]{Corollary}
\newtheorem{proposition}[theorem]{Proposition}

\theoremstyle{definition}
\newtheorem*{remark}{Remark}

\numberwithin{equation}{section}

\definecolor{arcColor}{HTML}{0A4A8A} 
\tikzstyle{mesh}=[pattern=north east lines, pattern color=gray!70, draw=gray]
\tikzstyle{arc1}=[draw, line width=1.2, color=arcColor]
\tikzstyle{arc2}=[line width=1, color=arcColor!25]

\def\abs#1{\lvert#1\rvert}
\def\Av{\mathcal{S}}

\DeclareMathOperator\id{id}
\DeclareMathOperator\height{height}

\def\R{\rule[-1ex]{0ex}{3.6ex}}

\definecolor{red1}{HTML}{B02400}

\usetikzlibrary{decorations.pathreplacing}

\newcommand\dyckpath[3]{
\def\diam{0.08}
\begin{scope}
	\draw[help lines] (#1) -- ++(#2*2,0);
	\draw[line width=1pt] (#1) foreach \dir in {#3}{ -- ++(\dir*90-45:1.41)};
	\draw[fill] (#1) circle (\diam);
	\draw[fill] (#1) foreach \dir in {#3}{ ++(\dir*90-45:1.41) circle (\diam)};
\end{scope}
}

\newcommand{\perm}[1]{%
	\def\Perm{#1}
	\StrSubstitute{\Perm}{,}{\,}
}

\newcommand{\drawarc}[3]{%
\def\sf{#1}
\def\n{#2}
\tikz[scale=\sf,baseline=0]{%
	\foreach \i in {1,...,\n} {\draw[fill,arc1] (\i,0) circle (0.06);}
	\foreach \x/\y in {#3} {%
		\pgfmathsetmacro\s{\x+\y}
		\pgfmathsetmacro\d{\y-\x}
		\pgfmathsetmacro\eps{\d^0.7}
		\draw[arc1] (\x,0) parabola bend (0.5*\s,0.25*\eps) (\y,0);
	}
	 \foreach \i in {1,...,\n}{\node[below=2pt] at (\i,0) {\small \i};}
}
}

\newcommand{\drawarcPerm}[4]{%
\def\sf{#1}
\def\n{#2}
\tikz[scale=\sf,baseline=0]{%
	\foreach \x/\y in {#3} {%
		\pgfmathsetmacro\s{\x+\y}
		\pgfmathsetmacro\d{\y-\x}
		\pgfmathsetmacro\eps{\d^0.7}
		\draw[arc1] (\x,0) parabola bend (0.5*\s,0.25*\eps) (\y,0);
	}
	\foreach \x/\y in {#4} {%
		\pgfmathsetmacro\s{\x+\y}
		\pgfmathsetmacro\d{\y-\x}
		\pgfmathsetmacro\eps{\d^0.7}
		\draw[arc2,<-] (\x,0) parabola bend (0.5*\s,-0.2*\eps) (\y,0);
	}
	 \foreach \i in {1,...,\n}{ \node[below=0pt,gray!40] at (\i,0) {\tiny \i};}
	\foreach \i in {1,...,\n} { \draw[fill,arc1] (\i,0) circle (0.06);}
	\foreach \x/\y in {#3}{ \node[above=4pt] at (\x,0) {\scriptsize \y};}
	\foreach \x/\y in {#4}{ \node[below=6pt] at (\y,0) {\scriptsize \x};}
}
}

\newcommand{\drawPoly}[2]{%
\def\sf{#1}
\hskip3pt
\tikz[scale=\sf,baseline=0]{%
	\foreach [count=\i] \b/\h in {#2}{
		\draw[thick] (\i,\b) grid (\i+1,\b+\h);
	}
}
}

\begin{document}

\title{Odd-indexed Fibonacci numbers via pattern-avoiding permutations}
\author[Gil]{Juan B. Gil}
\address{Penn State Altoona\\ 3000 Ivyside Park\\ Altoona, PA 16601}
\email{jgil@psu.edu}

\author[Xu]{Felix H. Xu}
\email{fxu8103@gmail.com}

\author[Zhu]{William Y. Zhu}
\email{wyzhu576@gmail.com}

\begin{abstract}
In this paper, we consider several combinatorial problems whose enumeration leads to the odd-indexed Fibonacci numbers, including certain types of Dyck paths, block fountains, directed column-convex polyominoes, and set partitions with no crossings and no nestings. Our goal is to provide bijective maps to pattern-avoiding permutations and derive generating functions that track certain positional statistics at the permutation level.
\end{abstract}

\maketitle

%%%%%%%%%%%%%%%%%%%%%%%%%%%%%%%%%%%%%%%%%%%%
\section{Introduction}
\label{sec:intro}

The Fibonacci sequence, defined by $F_1 = 1$, $F_2 = 1$, and $F_n = F_{n-1} + F_{n-2}$ for $n>2$, is one of the most common sequences in mathematics. They appear directly or indirectly (e.g.\ by means of the golden ratio) in various areas of mathematics as well as other fields like computer science, physics, and biology. In this paper, we are interested in the bisecting subsequence $(F_{2n-1})_{n\in\mathbb{N}}$. If we let $a_n=F_{2n-1}$, we then have $a_1=1$, $a_2=2$, and
\begin{equation}\label{eq:fiboRecurrence}
 a_n = 3a_{n-1} - a_{n-2} \;\text{ for } n\ge 3. 
\end{equation}
This sequence, starting with $1, 2, 5, 13, 34, 89, 233, 610, 1597, 4181, 10946, 28657, \dots$, is listed as sequence A001519 in the OEIS \cite{oeis} and has many interesting combinatorial interpretations on its own. We will focus on the following four: 
\begin{itemize}
\renewcommand\labelitemi{$\triangleright$}
\item Dyck paths of height at most 3
\item Block fountains of coins 
\item Set partitions with no crossings and no nestings
\item Directed column-convex polyominoes
\end{itemize}
Our goal here is to bijectively connect each of these combinatorial families to a corresponding class of pattern-avoiding permutations. It is known that there are nine symmetric classes of permutations, classically avoiding one pattern of length 3 and one of length 4, that are all enumerated by the odd-indexed Fibonacci numbers, see Atkinson~\cite{Atk99} or West~\cite{West96}. 

\smallskip
Let us briefly review some terminology. A permutation on $[n]=\{1,\dots,n\}$ is a one-to-one function $\sigma:[n]\to[n]$ that can be written as $\sigma=\sigma_1\sigma_2\cdots\sigma_n$, where $\sigma_i = \sigma(i)$ is the image of $i$ under $\sigma$.\footnote{This way of writing permutations is referred to as {\em one-line notation}.} We let $\Av_n$ denote the set of permutations on $[n]$ and let $\id_n=12\cdots n$.

Let $\tau\in\Av_k$. A permutation $\sigma$ is said to contain the pattern $\tau$ if it has a subsequence $\sigma(i_1),\dots,\sigma(i_k)$ whose elements are in the same relative order as those in $\tau$. For a set $P$ of patterns, we let $\Av_n(P)$ denote the set of permutations on $[n]$ that avoid (do not contain) any pattern $\tau\in P$. For example, the permutation $\sigma = \perm{3,4,1,2,5,6}$ belongs to $\Av_n(132,321)$, but it has multiple occurrences of the other patterns of length $3$. For instance, $\sigma$ contains two occurrences of the pattern $\tau=231$, namely $(3,4,1)$ and $(3,4,2)$. 

A permutation $\sigma$ is said to contain the {\em vincular} pattern $\underline{21}43$ if it contains a $2143$ pattern at positions $i_1< i_2 < i_3 < i_4$ with the additional condition that $i_2-i_1=1$. That is, in one-line notation, the values $\sigma(i_1)$ and $\sigma(i_2)$ must be adjacent. For example, the permutation $\sigma=\perm{3,4,1,6,2,5}$ contains two $2143$ patterns, $(3,1,6,5)$ and $(4,1,6,5)$, but only the latter is a vincular $\underline{21}43$ pattern.

For more on pattern-avoiding permutations, we refer to the book by Kitaev~\cite{Kitaev11}.

\smallskip
The results of the paper are presented in four sections:
\begin{itemize}
\renewcommand\labelitemi{--}
\item Section~\ref{sec:321_4123} $\leadsto$ $\Av_n(321,4123)$ and their connection to Dyck paths of semilength $n$ and height at most 3. In this section we also review some terminology and recall a useful characterization of $321$-avoiding permutations.
\item Section~\ref{sec:321_2143} $\leadsto$ $\Av_n(321,\underline{21}43)$ and their connection to block $n$-fountains. 
\item Section~\ref{sec:321_3412} $\leadsto$ $\Av_n(321,3412)$, also known as Boolean permutations. Using suitable arc diagrams on $n$ nodes, we provide a simple bijection between Boolean permutations on $[n]$ and noncrossing, nonnesting set partitions of $[n]$.
\item Section~\ref{sec:231_3124} $\leadsto$ $\Av_n(231,3124)$ and their connection to directed column-convex polyominoes of area $n$. 
\end{itemize}

Each section is similarly structured. We let $g_n$ be the cardinality of the corresponding set of permutations of size $n$ and show that the sequence satisfies the recurrence relation \eqref{eq:fiboRecurrence}. We also consider subsets tracking either the position of the $1$ or the position of the $n$ in each permutation and derive corresponding enumerative triangles together with their bivariate generating functions. Finally, we discuss the combinatorial family associated with the permutations and provide an explicit bijection. All of our bijective maps are constructive and reveal some information about the structure of the permutations under consideration.

%%%%%%%%%%%%%%%%%%%%%%%%%%%%%%%%%%%%%%%%%%%%
\section{Permutations avoiding 321 and 4123}
\label{sec:321_4123}

Before we dive into the content of this section, let us review some notation and discuss some properties of $321$-avoiding permutations. Given two permutations $\pi$ and $\sigma$ of sizes $m$ and $n$, respectively, their {\em direct sum} $\pi\oplus\sigma$ is the permutation of size $m+n$ consisting of $\pi$ followed by a shifted copy of $\sigma$. Similarly, their {\em skew sum} $\pi\ominus\sigma$ is the permutation consisting of $\sigma$ preceded by a shifted copy of $\pi$. For example, $312\oplus 21 = 31254$ and $312\ominus 21 = 53421$. 

For a permutation $\sigma$, the value $\sigma(j)$ is called a {\em left-to-right maximum} of $\sigma$ if $\sigma(j)>\sigma(i)$ for every $i<j$. It is straightforward to verify that a $321$-avoiding permutation is uniquely determined by the positions and values of its left-to-right maxima. Indeed, suppose 
\[ 1=i_1<i_2<\dots<i_k\le n \text{ with } k\le n, \text{ and } 1\le v_1<v_2<\dots<v_k=n. \]
Let $A=\{\ell_1,\dots,\ell_{n-k}\}$ be the complement of $\{i_1, \dots, i_k\}$ in $[n]$ and let $B=\{u_1,\dots,u_{n-k}\}$ be the complement of $\{v_1,\dots,v_k\}$, where $\ell_1< \dots <\ell_{n-k}$ and $u_1<\dots<u_{n-k}$. Let $\sigma$ be the permutation on $[n]$ defined by
\begin{align*}
  \sigma(i_j) &= v_j \text{ for } j\in \{1,\dots,k\}, \\
  \sigma(\ell_j) &=u_j \text{ for } \ell_j\in A, u_j\in B, \text{ and } j \in \{1,\dots,n-k\}.
\end{align*}
By construction, $\sigma$ consists of two increasing sequences and is therefore $321$-avoiding.

\medskip
We proceed with a simple proof of the following known statement.

\begin{proposition}
$\abs{\Av_n(321,4123)} = F_{2n-1}$.
\end{proposition}
\begin{proof}
Let $g_n = \abs{\Av_n(321,4123)}$. Clearly, $g_1=1$ and $g_2=2$. For $n\ge 3$, every permutation $\sigma\in \Av_n(321,4123)$ must have its largest element in one of its last three positions. 

Every $\sigma$ having $n$ at either the last or second-to-last position, can be obtained from a unique element of $\Av_{n-1}(321,4123)$ by inserting $n$ into the corresponding position. Thus, there are $g_{n-1}+g_{n-1}$ such permutations. However, inserting $n$ into the third-to-last position of a permutation of size $n-1$ does not work if the permutation ends with a descent (it would create a forbidden $321$ pattern), so we need to remove those. Now, since a permutation in $\Av_{n-1}(321,4123)$ ending with a descent must have its largest element in the second-to-last position, we conclude that there are $g_{n-2}$ such permutations. Therefore, there is a total of $g_{n-1} - g_{n-2}$ elements of $\Av_{n}(321,4123)$ having $n$ in the third-to-last position.

In conclusion, $g_1=1$, $g_2=2$, and $g_n = 3g_{n-1} - g_{n-2}$ for $n\ge 3$. Hence $g_n=F_{2n-1}$.
\end{proof}

We now consider the set $\Av_n^{k\mapsto1}(321,4123)$ of permutations in $\Av_n(321,4123)$ such that $\sigma(k)=1$. In one-line notation this means that entry $1$ is at position $k$. Note that for every $n>1$, the set $\Av_n^{n\mapsto1}(321,4123)$ consists of the single permutation $\sigma=\id_{n-1}\ominus 1$.

\begin{proposition}
\label{prop:321_4123Positional}
For $1\le k<n$, we have $\abs{\Av_n^{k\mapsto1}(321,4123)} = k\cdot \abs{\Av_{n-k}(321,4123)}$. A few terms are listed in Table~\ref{tab:triangle321_4123}. The function $g(x,t) = \sum\limits_{n\ge 1}\sum\limits_{k=1}^n \abs{\Av_n^{k\mapsto1}(321,4123)}\, t^k x^n$ satisfies
\[ g(x,t) =  \frac{tx}{1-tx} + \frac{tx}{(1-tx)^2}\cdot \frac{x-x^2}{1-3x+x^2}. \]
\end{proposition}

\begin{proof}
Every permutation $\sigma\in\Av_n(321,4123)$ with $\sigma(1)=1$ must be of the form $\sigma=1\oplus\sigma'$ with $\sigma'\in\Av_{n-1}(321,4123)$. Thus the statement is clear for $k=1$.

Suppose now that $1$ is at position $k$ with $1<k<n$. If there is an entry $a>k+1$ to the left of $1$, then there can be at most $k-2$ elements from the set $\{2,\dots,k+1\}$ to the left of $1$. So, at least two elements from that set, say $2\le b<c\le k+1$, would have to be to the right of $1$. But then $(a,1,c,b)$ would make a 321 pattern and $(a,1,b,c)$ would make a $4123$ pattern. Since all $k-1$ entries to the left of $1$ must form an increasing sequence, and we now have that the sequence must be made of elements from the set $\{2,\dots, k +1\}$, there are $\binom{k}{k-1}=k$ ways to choose such a sequence. The remaining $n-k$ entries, one from $\{2,\dots,k+1\}$ and $n-k-1$ from the set $\{k+2,\dots,n\}$, can be placed to the right of $1$ in any order as long as their reduced permutation avoids the patterns $321$ and $4123$. There are $\abs{\Av_{n-k}(321,4123)}$ such arrangements, and so $\abs{\Av_n^{k\mapsto1}(321,4123)} = k\cdot \abs{\Av_{n-k}(321,4123)}$.

To derive $g(x,t)$, we isolate the $k=n$ case and use the above formula to get
\[ g(x,t) = \sum_{n=1}^\infty t^n x^n + \sum_{n=2}^\infty \sum_{k=1}^{n-1} k g_{n-k} t^k x^n, \]
where $g_{n-k} = \abs{\Av_{n-k}(321,4123)} = F_{2(n-k)-1}$. Then,
\begin{align*}
 g(x,t) &= \frac{tx}{1-tx} + \sum_{k=1}^\infty \sum_{n=k+1}^{\infty} k g_{n-k} (tx)^k x^{n-k} \\
 &= \frac{tx}{1-tx} + tx\sum_{k=1}^\infty k(tx)^{k-1} \sum_{n=k+1}^{\infty} g_{n-k} x^{n-k} \\
 &= \frac{tx}{1-tx} + \frac{tx}{(1-tx)^2} \sum_{m=1}^{\infty} g_{m} x^{m} \\
 &= \frac{tx}{1-tx} + \frac{tx}{(1-tx)^2}\cdot \frac{x-x^2}{1-3x+x^2},
\end{align*}
using the fact that $\sum\limits_{m=1}^{\infty} g_{m} x^{m} = \dfrac{x-x^2}{1-3x+x^2}$.
\end{proof}

\begin{table}[ht]
\begin{tabular}{c|cccccccc}
$n\setminus k$ & 1 & 2 & 3 & 4 & 5 & 6 & 7 & 8 \\[2pt] \hline 
\R 1 & 1 &&&&&&& \\
2 & 1& 1 &&&&&& \\
3 & 2 & 2 & 1 &&&&& \\
4 & 5 & 4 & 3 & 1 &&&& \\
5 & 13 & 10 & 6 & 4 & 1 &&& \\
6 & 34 & 26 & 15 & 8 & 5 & 1 && \\
7 & 89 & 68 & 39 & 20 & 10 & 6 & 1 & \\
8 & 233 & 178 & 102 & 52 & 25 & 12 & 7 & 1
\end{tabular}
\bigskip 
\caption{Triangle for $\abs{\Av_n^{k\mapsto1}(321,4123)}$ and $\abs{\Av_n^{k\mapsto1}(231,3124)}$.  See \cite[A105292]{oeis}.}
\label{tab:triangle321_4123}
\end{table}

\subsection*{Dyck paths of height at most 3}
A Dyck path of semilength $n$ is a lattice path from $(0,0)$ to $(2n,0)$ using up-steps $(1,1)$, down-steps $(1,-1)$, and never going below the $x$-axis. An up-step immediately followed by a down-step is called a {\em peak}. The {\em height} of a Dyck path $D$, denoted $\height(D)$, is the height of its highest peak. 

There are several bijections between $\Av_n(321)$ and the set of Dyck paths of semilength $n$. The one that best fits our needs is a slightly modified version of a map given by Krattenthaler~\cite{Kratt01}, let's call it $\varphi_K$. Suppose $\sigma\in\Av_n(321)$ has left-to-right maxima $\sigma(i_1),\dots,\sigma(i_k)$, where $1=i_1<\cdots<i_k$. Reading the permutation $\sigma$ from left to right, we construct the Dyck path $D_{\sigma}=\varphi_K(\sigma)$ as follows. Start with $\sigma(i_1)$ up-steps, followed by $i_2-i_1$ down-steps. For every subsequent $j\in\{2,\dots,k\}$, we draw $\sigma(i_j)-\sigma(i_{j-1})$ up-steps, followed by $i_{j+1}-i_j$ down-steps (with the convention $i_{k+1}=n+1$).
 
Observe that the height of the $j$th peak of $D_\sigma$ is
\[ \sigma(i_1)-(i_2-i_1)+(\sigma(i_2)-\sigma(i_1))-(i_3-i_2) + \dots +(\sigma(i_j)-\sigma(i_{j-1}))=\sigma(i_j)-i_j + 1. \]
In particular, we have
\begin{equation} \label{eq:height}
\height(D_\sigma)>3 \;\text{ if and only if }\; \sigma(i_j) > i_j +2 \;\text{ for some } j\in\{1,\dots,k\}.
\end{equation}

For example, the permutation $\sigma=\perm{2,4,5,1,3}$ corresponds to the Dyck path
\medskip
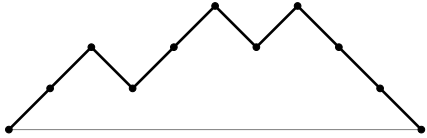
\begin{center}
\tikz[scale=0.55]{\dyckpath{0,0}{5}{1,1,0,1,1,0,1,0,0,0}}
\end{center}
having one peak of height $2$ and two peaks of height $3$.

\begin{proposition}
There is a bijection between $\Av_n(321,4123)$ and the set of Dyck paths of semilength $n$ and height at most 3.
\end{proposition}

\begin{proof}
This is a direct consequence of the above bijection $\varphi_K$. In fact, if $D_\sigma=\varphi_K(\sigma)$ has $\height(D_\sigma)>3$, then by \eqref{eq:height} we have $\sigma(i_j) > i_j +2$ for some left-to-right maximum value $\sigma(i_j)$. Now, since $\sigma$ has $i_j-1$ entries to the left of $\sigma(i_j)$ and since there are $i_j +2$ values smaller than $\sigma(i_j)$, at least three of these values must appear in increasing order to the right of $\sigma(i_j)$, creating a $4123$ pattern. In conclusion, if $\sigma$ is a $(321,4123)$-avoiding permutation, then $\height(D_\sigma)\le 3$.

Conversely, let $D$ be a Dyck path and let $\sigma=\varphi_K^{-1}(D)$. If $\sigma$ has a $4123$ pattern, then there is a left-to-right maximum $\sigma(i_j)$ (i.e.\ $\sigma(i_j)>\sigma(i)$ for every $i<i_j$) and there are values $a<b<c<\sigma(i_j)$ such that $i_j<\sigma^{-1}(a)<\sigma^{-1}(b)<\sigma^{-1}(c)$. In other words, we must have $\sigma(i_j)>(i_j-1)+3=i_j+2$, which by \eqref{eq:height} implies $\height(D)>3$. Therefore, if $\height(D)\le 3$, then $\sigma=\varphi_K^{-1}(D)$ avoids the patterns $321$ and $4123$.
\end{proof}

\begin{remark}
At the level of Dyck paths, the first column in Table~\ref{tab:triangle321_4123} gives the number of Dyck paths of height at most $3$ that start with an up-step followed by a down-step. For $k>1$, the table gives the counting of Dyck paths of height at most $3$ by the position of the first long descent (more than one consecutive down-steps).
\end{remark}

%%%%%%%%%%%%%%%%%%%%%%%%%%%%%%%%%%%%%%%%%%%%
\section{Permutations avoiding 321 and \texorpdfstring{\underline{21}43}{2143}}
\label{sec:321_2143}

In this section, we show that the elements of the set $\Av_n(321,\underline{21}43)$ are enumerated by the odd-indexed Fibonacci numbers. We also establish a bijection to block fountains.

\begin{proposition}
\label{prop:321_2143fibo}
$\abs{\Av_n(321,\underline{21}43)} = F_{2n-1}$.
\end{proposition}

\begin{proof}
Let $g_n = \abs{\Av_n(321,\underline{21}43)}$. Clearly, $g_1=1$ and $g_2=2$. For $n\ge 3$ we decompose $\Av_n(321,\underline{21}43)$ as a disjoint union of three sets, say $A_n^{1*}\cup A_n^{*n}\cup B_n$, where 
\begin{align*}
A_n^{1*} &= \{\sigma\in \Av_n(321,\underline{21}43): \sigma(1)=1\}, \\
A_n^{*n} &= \{\sigma\in \Av_n(321,\underline{21}43): \sigma(1)\not=1 \text{ and } \sigma(n)=n\}, \\
B_n &= \{\sigma\in \Av_n(321,\underline{21}43): \sigma(1)\not=1 \text{ and } \sigma(n)\not=n\}.
\end{align*}
Every element of $\Av_n(321,\underline{21}43)$ that starts with $1$ or ends with $n$ is of the form $1\oplus\sigma'$ or $\sigma'\oplus 1$, respectively, with $\sigma'\in \Av_{n-1}(321,\underline{21}43)$. In each case, there are $g_{n-1}$ such permutations. Moreover, every permutation in $\Av_n(321,\underline{21}43)$ that starts with $1$ and ends with $n$ is of the form $1\oplus\sigma''\oplus 1$ with $\sigma''\in \Av_{n-2}(321,\underline{21}43)$, so there are $g_{n-2}$ of those. Therefore, $\abs{A_n^{1*}} = g_{n-1}$ and $\abs{A_n^{*n}}=g_{n-1}-g_{n-2}$. Finally, there is a bijective map 
\begin{equation*}
\varphi: B_n \to \Av_{n-1}(321,\underline{21}43)
\end{equation*}
defined as follows. Let $\sigma\in B_n$ and let $\sigma(i_1),\dots,\sigma(i_m)$ be its left-to-right maxima listed in increasing order. Note that $i_1=1$, $\sigma(i_1)>1$,  $\sigma(i_m)=n$, and $i_m<n$. We let $\tau = \varphi(\sigma)$ be the unique permutation in $\Av_{n-1}(321)$ with left-to-right maxima $\tau(i_1),\dots,\tau(i_m)$, where $\tau(i_j) = \sigma(i_j)-1$ for $j\in\{1,\dots,m\}$. For example, 
\[ 2\,5\,1\,6\,3\,4 = \mathbf{2\,5}\,1\,\mathbf{6}\,3\,4 \mapsto \mathbf{1\,4}\,2\,\mathbf{5}\,3. \]
Suppose $\tau=\varphi(\sigma)$ contains a leftmost $\underline{21}43$ pattern, say $(b,a,d,c)$ with $a<b<c<d$. Then $b$ and $d$ are left-to-right maxima of $\tau$, entries $b$ and $a$ are adjacent, and $d-1$ must be to the right of $d$. Therefore, $\sigma$ must have $b+1$ and $d+1$ as left-to-right maxima at the same positions of $b$ and $d$ in $\tau$, and $d$ cannot be a left-to-right maximum of $\sigma$ because $d-1$ is not one for $\tau$. This means that $\sigma^{-1}(d) > \sigma^{-1}(d+1)$. Thus, if $a'$ is the element adjacent to the right of $b+1$ in $\sigma$, then $(b+1,a',d+1,d)$ is a $\underline{21}43$ pattern. In other words, if $\sigma$ avoids $\underline{21}43$, so does $\varphi(\sigma)$. 
This map is clearly bijective, which implies $\abs{B_n}=g_{n-1}$.

In conclusion, $g_1=1$, $g_2=2$, and $g_n = 3g_{n-1} - g_{n-2}$ for $n\ge 3$. Hence $g_n=F_{2n-1}$.
\end{proof}

\medskip
For the application discussed at the end of this section, it is meaningful to count the elements of $\Av_n(321,\underline{21}43)$ by the position of their largest entry. Let $\Av_n^{k\mapsto n}(321,\underline{21}43)$ be the set of permutations in $\Av_n(321,\underline{21}43)$ such that $\sigma(k)=n$. 

\begin{proposition}
Let $a_{n,k} = \abs{\Av_n^{k\mapsto n}(321,\underline{21}43)}$. We have 
\begin{align*}
 a_{n,1} &=1 \;\text{ for } n\ge 1, \\
 a_{n,n} &= F_{2n-3} \;\text{ for } n>1, \\[3pt] 
 a_{n,k} = a_{n-1,k-1} &+ a_{n-1,k} \;\text{ for } 1<k<n. 
\end{align*}
A few terms are listed in Table~\ref{tab:triangle321_2143}. The function $g(x,t) = \sum\limits_{n\ge 1}\sum\limits_{k=1}^n a_{n,k}\, t^k x^n$ satisfies
\[ g(x,t) =  \frac{tx(1-tx)(1-2tx)}{(1-x-tx)(1-3tx+t^2x^2)}. \]
\end{proposition}

\begin{proof}
The permutation $1\ominus\id_{n-1}$ is the only element of $\Av_n^{1\mapsto n}(321,\underline{21}43)$, so $a_{n,1}=1$. As discussed in the proof of Proposition~\ref{prop:321_2143fibo}, every element $\sigma\in\Av_n(321,\underline{21}43)$ can be uniquely obtained from an element $\sigma'\in\Av_{n-1}(321,\underline{21}43)$ as a direct sum, $1\oplus\sigma'$ or $\sigma'\oplus 1$, or by using the inverse map $\varphi^{-1}$. In particular, every permutation in $\Av_n^{n\mapsto n}(321,\underline{21}43)$ is of the form $\sigma'\oplus 1$, thus $a_{n,n} = \abs{\Av_{n-1}(321,\underline{21}43)} = F_{2n-3}$.

Now, for $1<k<n$, every $\sigma\in \Av_n^{k\mapsto n}(321,\underline{21}43)$ is either of the form $\sigma=1\oplus\sigma'$ with $\sigma'\in \Av_{n-1}^{k-1\mapsto n-1}(321,\underline{21}43)$ or $\sigma=\varphi^{-1}(\sigma')$ with $\sigma'\in \Av_{n-1}^{k\mapsto n-1}(321,\underline{21}43)$. Note that $\varphi$ and $\varphi^{-1}$ preserve the position of the largest element. In conclusion, $a_{n,k} = a_{n-1,k-1} + a_{n-1,k}$.

The expression for the generating function follows from routine algebraic manipulations, together with the generating function for the odd-indexed Fibonacci numbers.
\end{proof}

\begin{table}[ht]
\begin{tabular}{c|cccccccc}
$n\setminus k$ & 1 & 2 & 3 & 4 & 5 & 6 & 7 & 8 \\[2pt] \hline 
\R 1 & 1 &&&&&&& \\
2 & 1& 1 &&&&&& \\
3 & 1 & 2 & 2 &&&&& \\
4 & 1 & 3 & 4 & 5 &&&& \\
5 & 1 & 4 & 7 & 9 & 13 &&& \\
6 & 1 & 5 & 11 & 16 & 22 & 34 && \\
7 & 1 & 6 & 16 & 27 & 38 & 56 & 89 & \\
8 & 1 & 7 & 22 & 43 & 65 & 94 & 145 & 233
\end{tabular}
\bigskip 
\caption{Triangle for $\abs{\Av_n^{k\mapsto n}(321,\underline{21}43)}$. Reverse of A121460 in \cite{oeis}.}
\label{tab:triangle321_2143}
\end{table}

\subsection*{Block fountains of coins}
A {\em block $n$-fountain} of coins is an arrangement of coins in rows such that the bottom row consists of $n$ coins forming a contiguous block, and each higher row consists of a single contiguous block of coins where each coin touches exactly two coins from the row beneath it. For example:

\medskip
\begin{center}
\tikzstyle{coin}=[thick, color=black!60]
\begin{tikzpicture}[scale=1.3]
\begin{scope}
\foreach \x in {1,2,3,4,5,6} {\draw[coin] (0.5*\x,0) circle (0.25);}
\foreach \x in {3,5,7,9} {\draw[coin] (0.25*\x,0.44) circle  (0.25);}
\foreach \x in {6,8} {\draw[coin] (0.25*\x,0.88) circle  (0.25);}
\node[below=12pt] at (1.75,0) {Block $6$-fountain};
\end{scope}
\begin{scope}[xshift=120]
\foreach \x in {1,2,3,4,5} {\draw[coin] (0.5*\x,0) circle (0.25);}
\foreach \x in {3,7,9} {\draw[coin] (0.25*\x,0.44) circle  (0.25);}
\foreach \x in {8} {\draw[coin] (0.25*\x,0.88) circle  (0.25);}
\node[below=12pt] at (1.5,0) {Not a block fountain};
\end{scope}
\end{tikzpicture}
\end{center}

\medskip
As shown in Wilf~\cite[Section 2.1, Example 7]{Wgfology}, there are $F_{2n-1}$ block $n$-fountains. 

A coin in a block fountain is called a {\em peak} if it doesn't touch any coin in a higher row. A block $n$-fountain with maximum number of coins will be called a {\em triangular $n$-stack}. Note that a triangular stack has only one peak.

\begin{proposition}
There is a bijection between block $n$-fountains and $(321,\underline{21}43)$-avoiding permutations of size $n$.
\end{proposition}

\begin{proof}
We will provide an algorithm to go from block fountains to permutations, illustrating the steps with an example. Given a block $n$-fountain, proceed as follows:
\begin{enumerate}[$(i)$]
\item Going from left to right, label the coins on the bottom row with the elements of $[n]$ in increasing order.
\item Identify the peaks of the fountain. For every peak not at the bottom, draw diagonals (of slope $\pm\sqrt{3}$) from the center of the peak to the two coins on the bottom row along those diagonals. For example:
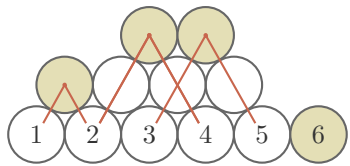
\begin{center}
\tikzstyle{coin}=[thick, color=black!60]
\tikzstyle{coin2}=[color=olive!25, thick, draw=black!60]
\begin{tikzpicture}[scale=1.5]
\foreach \x in {1,2,3,4,5} {\draw[coin] (0.5*\x,0) circle (0.25) node[black!80]{\small \x};}
\fill[coin2] (3,0) circle (0.25) node[black!80]{\small 6};
\foreach \x in {5,7,9} {\draw[coin] (0.25*\x,0.44) circle  (0.25);}
\fill[coin2] (0.75,0.44) circle (0.25);
\foreach \x in {6,8} {\fill[coin2] (0.25*\x,0.88) circle  (0.25);}
\draw[thick,red1!70] (0.56,0.1) -- (0.75,0.44) circle (0.01) -- (0.94,0.1); 
\draw[thick,red1!70] (1.06,0.1) -- (1.5,0.88) circle (0.01) -- (1.94,0.1); 
\draw[thick,red1!70] (1.06,0.1) -- (1.5,0.88) circle (0.01) -- (1.94,0.1); 
\draw[thick,red1!70] (1.56,0.1) -- (2,0.88) circle (0.01) -- (2.44,0.1); 
\end{tikzpicture}
\end{center}

\item If a bottom coin $i$ is a peak, we let $\sigma(i)=i$. Otherwise, if a peak at a higher row has diagonals pointing to the bottom coins $i$ and $j$ with $i<j$, we let $\sigma(i)=j$. 
\item If the fountain has $k$ peaks, the previous steps give positions $1=i_1<\dots<i_k$ and corresponding values $\sigma(i_1)<\dots<\sigma(i_k)=n$. We let $\sigma$ be the unique permutation in $\Av_n(321)$ having left-to-right maxima $\sigma(i_1), \dots, \sigma(i_k)$.
\end{enumerate}
 
For the above example, we get the permutation $\perm{2,4,5,1,3,6}$.

Suppose $\sigma$ contains a $\underline{21}43$ pattern at positions $(i,i+1,j,k)$ with $i+1<j<k$. This means that $\sigma(i)$ and $\sigma(j)$ are left-to-right maxima, and $\sigma(i+1)<\sigma(i)<\sigma(k)<\sigma(j)$. In particular, $j-i\ge 2$ and $\sigma(j)-\sigma(i)\ge 2$. Without loss of generality, assume $j-i = 2$ and $\sigma(j)-\sigma(i) = 2$. In this case, part of the fountain must have two peaks and be of the form:
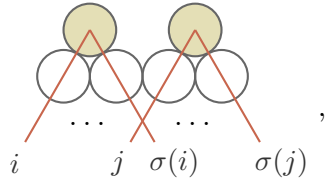
\begin{center}
\tikzstyle{coin}=[thick, draw=black!60]
\tikzstyle{coin2}=[color=olive!25, thick, draw=black!60]
\begin{tikzpicture}[scale=1.4,baseline=-15]
\foreach \x in {1,2,3,4} {\draw[coin] (0.5*\x,0) circle (0.25);}
\foreach \x in {3,7} {\fill[coin2] (0.25*\x,0.44) circle  (0.25);}
\draw[thick,red1!70] (0.75,0.44) -- +(240:35pt) node[below = 8pt, left = -2pt,black!80] {$i$}; 
\draw[thick,red1!70] (0.75,0.44) -- +(-60:35pt) node[below = 8pt, right = -6pt,black!80] {$\sigma(i)$}; 
\draw[thick,red1!70] (1.75,0.44) -- +(240:35pt) node[below = 8pt, left = -2pt,black!80] {$j$}; 
\draw[thick,red1!70] (1.75,0.44) -- +(-60:35pt) node[below = 8pt, right = -6pt,black!80] {$\sigma(j)$}; 
\node at (0.75,-0.45) {\large $\dots$};
\node at (1.75,-0.45) {\large $\dots$};
\end{tikzpicture},
\end{center}
which is not a block fountain because the coins at the top row are not contiguous. More generally, if $j-i>2$, the gap between the peaks would be longer, and if $\sigma(j)-\sigma(i) > 2$, the peak along diagonal $j$ would just be higher.

In conclusion, all permutations constructed as above avoid both $321$ and $\underline{21}43$.

The inverse map is obtained by reversing the algorithm. Given $\sigma \in \Av_n(321,\underline{21}43)$ let $i_1,\dots,i_k$ be the locations of its left-to-right maxima, listed in increasing order. Note that for every $j\in\{1,\dots,k\}$, we must have $\sigma(i_j)\ge i_j$. We start our fountain construction by placing a block of $n$ coins as base, labeling them with the elements of $[n]$ in increasing order. Then, for every $j$, we place the needed coins to make a triangular $(\sigma(i_j)-i_j+1)$-stack whose base are the coins labeled $i_j$ through $\sigma(i_j)$. As we did above, one can argue that any gap of coins in a row of the fountain would imply a $\underline{21}43$ pattern in $\sigma$. In other words, the fountain resulting from our construction must be a block $n$-fountain.
\end{proof}

%%%%%%%%%%%%%%%%%%%%%%%%%%%%%%%%%%%%%%%%%%%%
\section{Boolean permutations}
\label{sec:321_3412}

Boolean permutations are those that avoid the patterns $321$ and $3412$. They are known to be enumerated by the odd-indexed Fibonacci numbers, see Tenner \cite{Tenner07}.

\begin{proposition}
$\abs{\Av_n(321,3412)} = F_{2n-1}$.
\end{proposition}
\begin{proof}
Let $g_n = \abs{\Av_n(321,3412)}$. As stated before, it is clear that $g_1=1$ and $g_2=2$. Now, for $n\ge 3$ and $\sigma\in \Av_n(321,3412)$, we consider the disjoint cases $\sigma(1)=1$, $\sigma(2)=1$, or $\sigma(j)=1$ for some $j>2$. Every $\sigma$ with $\sigma(1)=1$ or $\sigma(2)=1$ can be uniquely obtained from a $\sigma'\in \Av_{n-1}(321,3412)$ by increasing all the entries of $\sigma'$ by one and inserting the $1$ either at position 1 or position 2, respectively. Thus there are $2g_{n-1}$ such permutations.

If $\sigma$ has the $1$ at position $j>2$, then $\sigma(1)=2$; otherwise entry $2$ would create either a $321$ pattern (if it is to the left of $1$, but not at position 1) or a $3412$ pattern (if it is to the right of $1$). Such a permutation can be uniquely obtained by inserting the $2$ at position 1 into a permutation from $\Av_{n-1}(321,3412)$ that does not start with $1$. There are $g_{n-1} - g_{n-2}$ permutations of this type.

In conclusion, $g_1=1$, $g_2=2$, and $g_n = 3g_{n-1} - g_{n-2}$ for $n\ge 3$. Hence $g_n=F_{2n-1}$.
\end{proof}

\medskip
Similar to Section~\ref{sec:321_2143}, we consider the set $\Av_n^{k\mapsto1}(321,3412)$ of permutations in $\Av_n(321,3412)$ having the $1$ in position $k$, i.e.\ $\sigma(k)=1$. 

\begin{proposition}
For $n> 1$, we have $\abs{\Av_n^{1\mapsto1}(321,3412)} = \abs{\Av_n^{2\mapsto1}(321,3412)} = F_{2n-3}$, and for $3\le k\le n$, we have $\abs{\Av_n^{k\mapsto1}(321,3412)} = \abs{\Av_{n-1}^{k-1\mapsto1}(321,3412)}$. Some terms are listed in Table~\ref{tab:triangle321_3412}. Moreover, the function $g(x,t) = \sum\limits_{n\ge 1}\sum\limits_{k=1}^n \abs{\Av_n^{k\mapsto1}(321,3412)}\, t^k x^n$ satisfies
\[ g(x,t) = \frac{tx-2tx^2+t^2x^3}{(1-tx)(1-3x+x^2)}. \]
\end{proposition}

\begin{proof}
This follows from the discussion in the proof of the previous proposition. First of all, every $\sigma\in \Av_{n}(321,3412)$ with $\sigma(1)=1$ or $\sigma(2)=1$ can be uniquely obtained from a permutation in $\Av_{n-1}(321,3412)$ by inserting the $1$ at the corresponding position. This leads to the claimed formula for $k=1$ and $k=2$. For $k\ge 3$, every element of $\Av_n^{k\mapsto1}(321,3412)$ can be uniquely obtained by inserting $2$ at position $1$ into an element of $\Av_{n-1}^{k-1\mapsto1}(321,3412)$. This proves the claimed formula for $k\ge 3$.

As in previous propositions, the claimed rational expression for $g(x,t)$ follows from straightforward algebraic manipulations.
\end{proof}

\begin{table}[ht]
\begin{tabular}{c|cccccccc}
$n\setminus k$ & 1 & 2 & 3 & 4 & 5 & 6 & 7 & 8 \\[2pt] \hline 
\R 1 & 1 &&&&&&& \\
2 & 1& 1 &&&&&& \\
3 & 2 & 2 & 1 &&&&& \\
4 & 5 & 5 & 2 & 1 &&&& \\
5 & 13 & 13 & 5 & 2 & 1 &&& \\
6 & 34 & 34 & 13 & 5 & 2 & 1 && \\
7 & 89 & 89 & 34 & 13 & 5 & 2 & 1 & \\
8 & 233 & 233 & 89 & 34 & 13 & 5 & 2 & 1
\end{tabular}
\bigskip 
\caption{Triangle for $\abs{\Av_n^{k\mapsto1}(321,3412)}$.}
\label{tab:triangle321_3412}
\end{table}

\subsection*{Noncrossing, nonnesting set partitions}
We finish this section with a discussion on noncrossing, nonnesting set partitions, and their representation via arc diagrams. Our goal is to provide a simple bijection between these combinatorial objects and the set of $(321,3412)$-avoiding permutations.

A partition of the set $[n]=\{1,\dots,n\}$ is a set of disjoint nonempty sets (called blocks) whose union is $[n]$. Every partition $\pi$ of $[n]$ can be represented by an arc diagram obtained by drawing an arc between each pair of integers that appear consecutively in the same block of $\pi$. For example, the partition $\pi = \{\{1,3\}, \{2,4,8\}, \{5,7\}, \{6\}\}$ can be represented as
\medskip
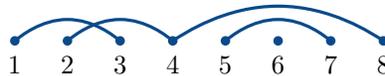
\begin{center}
\drawarc{0.7}{8}{1/3,2/4,4/8,5/7}
\end{center}
Two arcs $(i_1,j_1)$ and $(i_2,j_2)$ make a {\em crossing} if $i_1<i_2<j_1<j_2$, and they make a {\em nesting} if $i_1<i_2<j_2<j_1$. A {\em noncrossing/nonnesting partition} is a partition with no crossings/nestings. Moreover, a partition of $[n]$ is said to be {\em indecomposable} if no subset of its blocks is a partition of $[k]$ with $k<n$. In other words, an indecomposable partition is one whose corresponding arc diagram cannot be separated into two disjoint arc diagrams with consecutive nodes. The partition $\{\{1,3\}, \{2,4,8\}, \{5,7\}, \{6\}\}$, shown above, is an example of an indecomposable partition with one crossing and one nesting.

We let $\mathcal{P}_{\rm ncn}(n)$ be the set of partitions of $[n]$ that are both noncrossing and nonnesting, and let $\mathcal{P}_{\rm ncn}^{i}(n)$ be the subset of such partitions that are indecomposable. Note that every partition in $\mathcal{P}_{\rm ncn}^{i}(n)$ consists of a block of size $j$ for some $j\ge 2$, containing $1$ and $n$, together with $n-j$ singletons. It is known that $\abs{\mathcal{P}_{\rm ncn}(n)} = F_{2n-1}$, see Marberg~\cite[Example~4.2]{Mar13}.

\begin{proposition}
\label{prop:nonNestCrossBijection}
There is a bijection between indecomposable, noncrossing, nonnesting set partitions of $[n]$ and $(321,3412)$-avoiding indecomposable permutations on $[n]$.\footnote{A permutation is called {\em indecomposable} if it is not a direct sum of two nonempty permutations.}
\end{proposition}

\begin{proof}
First, we map $\{1\}\mapsto 1$ and $\{1,2\}\mapsto 21$. 

Let $n\ge 3$ and $\pi \in \mathcal{P}_{\rm ncn}^{i}(n)$, say $\pi = \{\{u_0,\dots,u_j\}, \{s_1\},\dots,\{s_k\}\}$ with $j+1+k=n$, $u_0=1$, $u_j=n$, $u_0<u_1<\cdots<u_{j}$, and $2\le s_1<\cdots<s_k\le n-1$. We let $\sigma_\pi$ be the unique element of $\Av_n(321)$ with left-to-right maxima $u_1,\dots,u_j$ at positions $1,u_1,\dots,u_{j-1}$, respectively. In other words, if we let $s_0=1$, then
\begin{gather*}
\sigma_\pi(1) = u_1, \;\; \sigma_\pi(n) = s_k, \\[2pt]
\begin{aligned}
\sigma_\pi(u_{i-1}) &= u_i \text{ for } i\in\{2,\dots,j\}, \\
\sigma_\pi(s_i) &= s_{i-1} \text{ for } i\in\{1,\dots,k\}.
\end{aligned}
\end{gather*}
For example, the partition $\{\{1,2,4,5,8\},\{3\},\{6\},\{7\}\} \in \mathcal{P}_{\rm ncn}^{i}(8)$ corresponds in one-line notation to the permutation $\perm{2,4,1,5,8,3,6,7}$. This map can be visualized as follows:
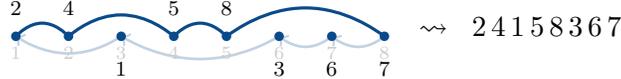
\begin{center}
\drawarcPerm{0.7}{8}{1/2,2/4,4/5,5/8}{1/3,3/6,6/7,7/8} $\leadsto\;\; \perm{2,4,1,5,8,3,6,7}$
\end{center}
The elements of $\mathcal{P}_{\rm ncn}^{i}(5)$ are listed in Table~\ref{tab:nonNesCrossBijection} together with their corresponding permutations.

Observe that the permutation $\sigma_\pi$ consists of a single cycle, $\sigma_\pi = (1\; u_1\, \cdots\, u_j \; s_k\, \cdots\, s_1)$, and is therefore indecomposable. 

If $\sigma_\pi$ were to contain a $3412$ pattern, then the $3$ and $4$ will have to be left-to-right maxima (otherwise $\sigma_\pi$ would have a $321$ pattern). Let $u_{i-1}$ and $u_i$ be nodes in the arc diagram of $\pi$ connected by a single arc:
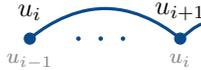
\begin{center}
\begin{tikzpicture}
	\foreach \i/\n in {1/u_{i-1},3/u_{i}} {%
		\draw[fill,arc1] (\i,0) circle (0.06);
		\node[below=2pt, gray] at (\i,0) {\scriptsize $\n$};
	}
	\foreach \i/\n in {1/u_{i},3/u_{i+1}} {\node[above=3pt] at (\i,0) {\small $\n$};}
	\draw[arc1] (1,0) parabola bend (2,0.4) (3,0);
	\draw[arc1] (3,0) to [out=35,in=200] (3.3,0.2);
	\node[thick,color=arcColor] at (2,0) {\huge \dots};
\end{tikzpicture}
\end{center}
By our construction of $\sigma_\pi$ (in one-line notation), there are $u_{i-1}-1$ values to the left of and smaller than $u_i$, and $u_i-u_{i-1}-1$ values of $\sigma_\pi$ strictly between the entries $u_i$ and $u_{i+1}$ (all smaller than $u_i$). Thus, there are $(u_{i-1}-1) + (u_i-u_{i-1}-1) = u_i-2$ numbers smaller than $u_i$ to the left of $u_{i+1}$. Therefore, there can be at most one element smaller than $u_i$ to the right of $u_{i+1}$, which means that a $3412$ pattern is not possible.

In conclusion, $\sigma_\pi$ is an indecomposable element of $\Av_n(321,3412)$. 

This map is clearly invertible. Given an indecomposable $\sigma\in\Av_n(321,3412)$ with left-to-right maxima $u_1,\dots,u_j$, listed in increasing order, we must have $u_1>1$, $\sigma^{-1}(u_1)=1$, $u_j=n$, and $\sigma^{-1}(u_j)<n$. We then construct an arc diagram with $n$ nodes, labeled $1$ through $n$ from left to right, by connecting the first and last nodes with consecutive arcs passing through the nodes labeled $1, u_1,\dots, u_{j}$. The resulting diagram gives the indecomposable, noncrossing, nonnesting partition of $[n]$,
\[ \pi_\sigma=\{\{1,u_1,\dots,u_j\},\{s_1\},\dots,\{s_{n-j-1}\}\}, \]
where $s_1,\dots, s_{n-j-1}$ are the labels of the $n-j-1$ isolated nodes of the diagram.
\end{proof}

\begin{table}[ht]
\def\eps{0.4}
\scriptsize
\begin{tabular}{c|c|c}
$\pi$ & Diagram & $\sigma_\pi$ \\[2pt] \hline\hline
$\{\{1,2,3,4,5\}\}$ & \drawarcPerm{\eps}{5}{1/2,2/3,3/4,4/5}{1/5} & \perm{2,3,4,5,1} \\
$\{\{1,2,3,5\},\{4\}\}$ & \drawarcPerm{\eps}{5}{1/2,2/3,3/5}{1/4,4/5} & \perm{2,3,5,1,4} \\
$\{\{1,2,4,5\},\{3\}\}$ & \drawarcPerm{\eps}{5}{1/2,2/4,4/5}{1/3,3/5} & \perm{2,4,1,5,3} \\
$\{\{1,3,4,5\},\{2\}\}$ & \drawarcPerm{\eps}{5}{1/3,3/4,4/5}{1/2,2/5} & \perm{3,1,4,5,2} \\
\end{tabular}
\hspace{3em}
\begin{tabular}{c|c|c}
$\pi$ & Diagram & $\sigma_\pi$ \\[2pt] \hline\hline
$\{\{1,2,5\},\{3\},\{4\}\}$ & \drawarcPerm{\eps}{5}{1/2,2/5}{1/3,3/4,4/5} & \perm{2,5,1,3,4} \\
$\{\{1,3,5\},\{2\},\{4\}\}$ & \drawarcPerm{\eps}{5}{1/3,3/5}{1/2,2/4,4/5} & \perm{3,1,5,2,4} \\
$\{\{1,4,5\},\{2\},\{3\}\}$ & \drawarcPerm{\eps}{5}{1/4,4/5}{1/2,2/3,3/5} & \perm{4,1,2,5,3} \\
$\{\{1,5\},\{2\},\{3\},\{4\}\}$ & \drawarcPerm{\eps}{5}{1/5}{1/2,2/3,3/4,4/5} & \perm{5,1,2,3,4} \\
\end{tabular}
\bigskip 
\caption{Bijection between $\mathcal{P}_{\rm ncn}^{i}(5)$ and indecomposable elements of $ \Av_5(321,3412)$.}
\label{tab:nonNesCrossBijection}
\end{table}

\begin{corollary}
There is a bijection between noncrossing, nonnesting set partitions of $[n]$ and $(321,3412)$-avoiding  permutations on $[n]$.
\end{corollary}

This follows by applying the map from Proposition~\ref{prop:nonNestCrossBijection} to each indecomposable component. If $\pi$ has a decomposition $\pi_1 | \pi_2 | \cdots | \pi_k$, then we let $\sigma_\pi =  \sigma_{\pi_1} \oplus \sigma_{\pi_2} \oplus \cdots \oplus \sigma_{\pi_k}$. 

For example, for $\pi = \{\{1,2,4\},\{3\},\{5\},\{6,8\},\{7\}\}$, we have $\pi=\pi_1|\pi_2|\pi_3$ with 
\[ \pi_1 = \{\{1,2,4\},\{3\}\}, \; \pi_2=\{1\}, \text{ and } \pi_3 = \{\{1,3\},\{2\}\}. \]
Then $\sigma_{\pi_1} = \perm{2,4,1,3}$, $\sigma_{\pi_2} =1$, $\sigma_{\pi_3} = \perm{3,1,2}$, and $\sigma_\pi = \perm{2,4,1,3}\oplus1\oplus\perm{3,1,2} = \perm{2,4,1,3,5,8,6,7}$.
\[ \drawarcPerm{0.8}{8}{1/2,2/4,6/8}{1/3,3/4,6/7,7/8} \;\;\leadsto\;\; \perm{2,4,1,3,5,8,6,7}. \quad \]

\begin{remark}
If $\pi$ is a partition of $[n]$ with blocks $B_1,\dots,B_j$, we let 
\[ m(\pi)=\min\{\max(B_1),\dots,\max(B_j)\}. \]
The $k$th column of Table~\ref{tab:triangle321_3412} gives the number of partitions $\pi\in \mathcal{P}_{\rm ncn}(n)$ such that $m(\pi)=k$.  
\end{remark}

%%%%%%%%%%%%%%%%%%%%%%%%%%%%%%%%%%%%%%%%%%%%
\section{Permutations avoiding 231 and 3124}
\label{sec:231_3124}

As in previous sections, we start with a simple proof of the known fact that the elements of $\Av_n(231,3124)$ are counted by the odd-indexed Fibonacci numbers.

\begin{proposition}
\label{prop:231_3124fibo}
$\abs{\Av_n(231,3124)} = F_{2n-1}$.
\end{proposition}
\begin{proof}
Let $g_n = \abs{\Av_n(231,3124)}$. For $n\ge 3$ and $\sigma\in \Av_n(231,3124)$, we consider the disjoint cases $\sigma(1)=1$, $\sigma(1)=n$, or $1<\sigma(1)<n$. The first two types are of the form $1\oplus \sigma'$ and $1\ominus \sigma'$, respectively, for some $\sigma'\in \Av_{n-1}(231,3124)$. Thus there are $g_{n-1}$ permutations of each type for a combined total of $2g_{n-1}$.

Let $A_n = \{\tau\in \Av_n(231,3124): 1<\tau(1)<n \}$ and let $\sigma\in A_n$. Then, we must have
\[ \sigma(2) = \sigma(1)-1. \]
Indeed, $\sigma(1)<\sigma(2)$ would produce the $231$ pattern $(\sigma(1),\sigma(2),1)$, and if $\sigma(1)-\sigma(2)>1$, then there would exist $a\in[n]$ with $\sigma(1)>a>\sigma(2)$. In that case, either $(\sigma(1),n,a)$ would form a $231$ pattern or $(\sigma(1),\sigma(2),a,n)$ would form a $3124$ pattern. Since both patterns are avoided by $\sigma$, we conclude that $\sigma(2) = \sigma(1)-1$ holds.

This means that every permutation $\sigma\in A_n$ can be uniquely obtained from a permutation $\sigma'\in\Av_{n-1}(231,3124)$ with $\sigma'(1)<n-1$ (there are $g_{n-1} - g_{n-2}$ of those) by $21$-inflating its first entry. More precisely, for every $\sigma'$ as above, we let $\sigma\in\Av_{n}(231,3124)$ be defined by letting $\sigma(1) = \sigma'(1)+1$, and for $j\in\{2,\dots,n\}$, we let
\[ 
\sigma(j) = \begin{cases}
	\sigma'(j-1) &\text{if } \sigma'(j-1) \le \sigma'(1), \\
	\sigma'(j-1)+1 &\text{if } \sigma'(j-1) > \sigma'(1).
	\end{cases}
\]
For example, the three permutations in $\Av_{4}(231,3124)$ not starting with $1$ or $4$ correspond to the three elements of $\Av_{3}(231,3124)$ not starting with $3$:
\[ \perm{1,2,3} \leadsto \perm{2,1,3,4}, \quad \perm{1,3,2} \leadsto \perm{2,1,4,3}, \quad \perm{2,1,3} \leadsto \perm{3,2,1,4}. \]
In conclusion, $g_n = 3g_{n-1} - g_{n-2}$ for $n\ge 3$. Since $g_1=1$ and $g_2=2$, we have $g_n=F_{2n-1}$.
\end{proof}

\medskip
We now consider the subset $\Av_n^{k\mapsto1}(231,3124)\subset \Av_n(231,3124)$ of permutations having the $1$ in position $k$. Observe that the set $\Av_n^{n\mapsto1}(231,3124)$ only contains the decreasing permutation $\sigma = n\, (n-1) \cdots 1$.

\begin{proposition}
\label{prop:231_3124Pos1}
For $1\le k < n$, we have $\abs{\Av_n^{k\mapsto1}(231,3124)} = k\cdot \abs{\Av_{n-k}(231,3124)}$. As in Proposition~\ref{prop:321_4123Positional}, the function $g(x,t) = \sum\limits_{n\ge 1}\sum\limits_{k=1}^n \abs{\Av_n^{k\mapsto1}(231,3124)}\, t^k x^n$ satisfies
\[ g(x,t) =  \frac{tx}{1-tx} + \frac{tx}{(1-tx)^2}\cdot \frac{x-x^2}{1-3x+x^2}. \]
\end{proposition}

\begin{proof}
Every $\sigma\in\Av_n(231,3124)$ that starts with $1$ can be written as $\sigma=1\oplus\sigma'$, where $\sigma'$ is an element of $\Av_{n-1}(231,3124)$. Thus $\abs{\Av_n^{1\mapsto1}(231,3124)} = \abs{\Av_{n-1}(231,3124)}$, as claimed.

Suppose that $\sigma$ is such that $\sigma(k)=1$ for $1<k<n$. With the same argument used in the proof of Proposition~\ref{prop:231_3124fibo}, we deduce that, if $1<\sigma(1)<n$, then $\sigma$ must start with the elements of $[k]$ in decreasing order. This means, $\sigma = \delta_k \oplus \sigma''$, where $\delta_k$ is the decreasing permutation on $[k]$ and $\sigma''\in\Av_{n-k}(231,3124)$.

On the other hand, if $\sigma(k)=1$ for $1<k<n$ and $\sigma(1)=n$, then the first $k$ entries of $\sigma$ must consist of two decreasing sequences of consecutive numbers. Indeed, if there were two gaps in the values of $\sigma$ to the left of $1$, there would exist entries $a,b$ appearing to the right of $1$ such that $n>a>\sigma(i)>b> \sigma(i+1)$ for some $i\in\{2,\dots,k-1\}$. Now, if $a$ is to the left of $b$, then $(\sigma(i),a,b)$ would form a $231$ pattern, and if $a$ is to the right of $b$, then $(\sigma(i),1,b,a)$ would form a $3124$ pattern. In other words, in this case, $\sigma$ has exactly one gap in its first descending run. Hence $\sigma = \delta_{j}\ominus (\delta_{k-j}\oplus\sigma'')$ for some $j\in\{1,\dots,k-1\}$, where $\delta_{j}$ and $\delta_{k-j}$ are the decreasing permutations on $[j]$ and $[k-j]$, and $\sigma''\in\Av_{n-k}(231,3124)$:
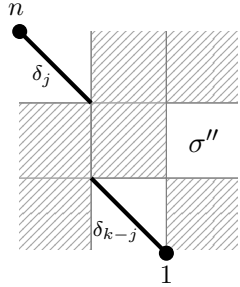
\begin{center}
\begin{tikzpicture}
\fill (0.05,2.95) circle (0.1);
\fill (2,0) circle (0.1);
\node[above=1pt] at (0,3) {\small $n$};
\node[below=1pt] at (2,0) {\small $1$};
\clip (0.05,0.05) rectangle (2.95,2.95);
\draw[gray] (0,0) grid (3,3);
\draw[mesh] (0,0) rectangle (1,2);
\draw[mesh] (1,1) rectangle (2,3);
\draw[mesh] (2,0) rectangle (3,1);
\draw[mesh] (2,2) rectangle (3,3);
\draw[ultra thick] (0,3) -- (1,2);
\draw[ultra thick] (1,1) -- (2,0);
\fill (2,0) circle (0.1);
\node[below=2pt] at (0.35,2.7) {\scriptsize $\delta_j$};
\node[below=1pt] at (1.33,0.6) {\scriptsize $\delta_{k-j}$};
\node at (2.5,1.5) {$\sigma''$};
\end{tikzpicture}
\end{center}
In conclusion, for $1<k<n$, every $\sigma''\in \Av_{n-k}(231,3124)$ leads to $k$ different permutations in $\Av_n^{k\mapsto1}(231,3124)$. This finishes the proof.
\end{proof}

\begin{remark}
Let $\delta_0$ be the empty permutation. In the previous proof, we have shown that every $\sigma\in\Av_n(231,3124)$ with $\sigma(k_1)=1$ is of the form $\sigma = \delta_{j_1}\ominus (\delta_{k_1-j_1}\oplus\sigma_{k_1})$ for some $j_1\in\{0,1,\dots,k_1-1\}$ and $\sigma_{k_1}\in \Av_{n-k_1}(231,3124)$. The argument can certainly be iterated. If $\sigma_{k_1}$ has the $1$ at position $k_2$, then $\sigma_{k_1} = \delta_{j_2}\ominus (\delta_{k_2-j_2}\oplus\sigma_{k_2})$ for some $j_2\in\{0,1,\dots,k_2-1\}$ and $\sigma_{k_2}\in \Av_{n-k_1-k_2}(231,3124)$, and therefore
\[ \sigma = \delta_{j_1}\ominus (\delta_{k_1-j_1}\oplus (\delta_{j_2}\ominus (\delta_{k_2-j_2}\oplus\sigma_{k_2}))). \]
If $\sigma$ has $d+1$ descending runs, then the above process can be iterated $d$ times until $\sigma_{k_{d}}$ is the decreasing permutation on $[n-k_1-\dots-k_d]$. In particular, this implies that all the descending runs of $\sigma$ consist of one or two sequences of consecutive numbers.
\end{remark}

\subsection*{Directed column-convex polyominoes}
We finish this section with a direct bijection between $\Av_n(231,3124)$ and the set of directed column-convex (DCC) polyominoes of area $n$. A {\em directed} polyomino is one that can be built by starting with a single cell and adding new cells on the right or on the top of an existing cell. A polyomino is called {\em column-convex} if every column consists of contiguous cells. The {\em area} of a polyomino is the number of its cells. An example of a DCC polyomino is shown in Figure~\ref{fig:9dccPoly}.
\begin{figure}[ht]
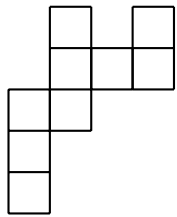

\drawPoly{0.55}{0/3,2/3,3/1,3/2}
\caption{Directed column-convex polyomino of area $9$.}
\label{fig:9dccPoly}
\end{figure}
It is known (see e.g.\ Delest \& Delucq \cite{DD92}, Barcucci et.~al.\ \cite{BPS93}) that the number of DCC polyominoes of area $n$ is the Fibonacci number $F_{2n-1}$.

\begin{proposition}
There is a bijection between the set of directed column-convex polyominoes of area $n$ and the set of $(231,3124)$-avoiding permutations of size $n$.
\end{proposition}

\begin{proof}
Let $P$ be a DCC polyomino with $k$ columns and $n$ cells. We will construct a corresponding permutation $\sigma_P$ in $\Av_n(231,3124)$, having $k$ descending runs. Every step of our algorithm will be illustrated with an example.
\begin{enumerate}[$(i)$]
\item First, going through the columns of $P$ from left to right, mark the cell that aligns with the bottom cell of the adjacent column to its right. For example, if $P$ is the polyomino from Figure~\ref{fig:9dccPoly}, we get:
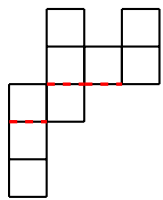
\begin{center}
\begin{tikzpicture}
\node at (0.63,1) {\drawPoly{0.5}{0/3,2/3,3/1,3/2}};
\foreach [count=\x] \b in {1,2,2} {
	\draw[line width=1.25,dashed,red] (0.5*\x-0.75,0.5*\b+0.25) -- (0.5*\x-0.25,0.5*\b+0.25);
	}
\end{tikzpicture}
\end{center}
\item Going through the columns from left to right, use the elements of $[n]$ in consecutive order (starting with 1) to label the cells of each column from the top down to the cell that was marked in the previous step. 
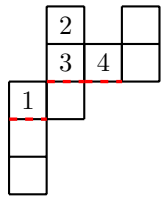
\begin{center}
\begin{tikzpicture}
\node at (0.63,1) {\drawPoly{0.5}{0/3,2/3,3/1,3/2}};
\foreach [count=\x] \b in {1,2,2} {
	\draw[line width=1.25,dashed,red] (0.5*\x-0.75,0.5*\b+0.25) -- (0.5*\x-0.25,0.5*\b+0.25);
	}
\node at (0,1) {\small $1$};
\node at (0.5,2) {\small $2$};
\node at (0.5,1.5) {\small $3$};
\node at (1,1.5) {\small $4$};
\end{tikzpicture}
\end{center}
\item Finally, going through the columns of $P$ from right to left, continue labeling the remaining cells in consecutive order from top to bottom. When all cells are filled with the elements of $[n]$, we build the permutation $\sigma_P$ from left to right by reading the numbers in each column from bottom to top as its descending runs.
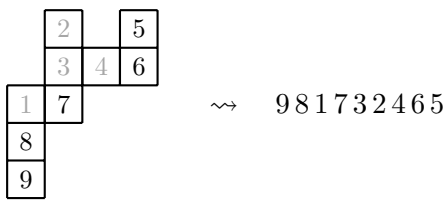
\begin{center}
\begin{tikzpicture}
\node at (0.63,1) {\drawPoly{0.5}{0/3,2/3,3/1,3/2}};
\node[gray!70] at (0,1) {\small $1$};
\node[gray!70] at (0.5,2) {\small $2$};
\node[gray!70] at (0.5,1.5) {\small $3$};
\node[gray!70] at (1,1.5) {\small $4$};
\node at (1.5,2) {\small $5$};
\node at (1.5,1.5) {\small $6$};
\node at (0.5,1) {\small $7$};
\node at (0,0.5) {\small $8$};
\node at (0,0) {\small $9$};
\node at (4,1) {$\leadsto\quad \perm{9,8,1,7,3,2,4,6,5}$};
\end{tikzpicture}
\end{center}
\end{enumerate}
More examples are shown in Figure~\ref{fig:4dccpPerm}. By construction, the resulting permutation has the nesting properties discussed in the remark after Proposition~\ref{prop:231_3124Pos1}. So, $\sigma_P\in \Av_n(231,3124)$. 

The above algorithm is reversible. Let $\sigma\in \Av_n(231,3124)$ and suppose $\sigma = w_1w_2\cdots w_k$, where $w_1,\dots,w_k$ are its descending runs. As discussed before, every $w_j$ is of the form $w_j = u_j\,v_j$, where $u_j$ and $v_j$ are words consisting of decreasing consecutive numbers (allowing $u_j$ to be the empty word). For every $j\in\{1,\dots,k\}$, we let $C_j$ be the polyomino consisting of $|w_j| = |u_j|+|v_j|$ cells in a single column. We now construct a corresponding polyomino $P_\sigma$ with $n$ cells by connecting $C_1,\dots,C_k$ as follows. Once column $C_j$ is placed, column $C_{j+1}$ will be attached to its right in such a way that the base cell of $C_{j+1}$ is adjacent to the $|v_j|$-th cell from the top of $C_j$. Clearly, $P_\sigma$ is a DCC polyomino of area $n$. 
\end{proof}

\begin{figure}[ht]
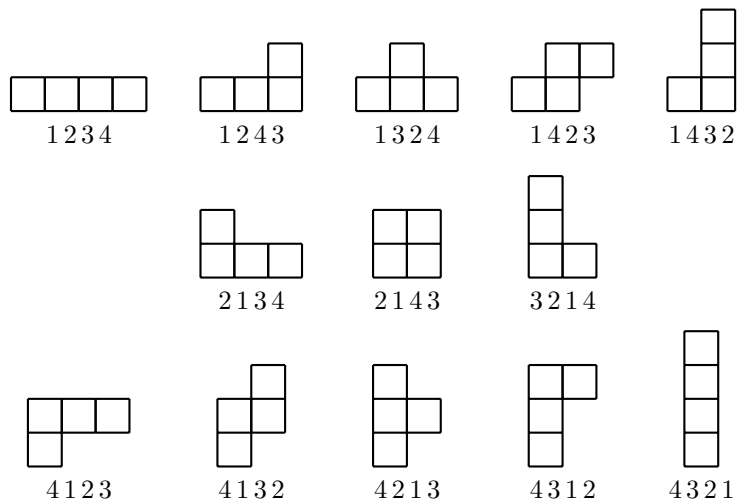

\def\eps{0.45}
\small
\begin{tabular}{ccccc}
\drawPoly{\eps}{0/1,0/1,0/1,0/1} &
\drawPoly{\eps}{0/1,0/1,0/2} &
\drawPoly{\eps}{0/1,0/2,0/1} &
\drawPoly{\eps}{0/1,0/2,1/1} & 
\drawPoly{\eps}{0/1,0/3}
\\
\perm{1,2,3,4} &
\perm{1,2,4,3} &
\perm{1,3,2,4} &
\perm{1,4,2,3} &
\perm{1,4,3,2}
\\[2ex]
& \drawPoly{\eps}{0/2,0/1,0/1} &
\drawPoly{\eps}{0/2,0/2} &
\drawPoly{\eps}{0/3,0/1} &
\\
& \perm{2,1,3,4} &
\perm{2,1,4,3} &
\perm{3,2,1,4} & 
\\[1ex]
\drawPoly{\eps}{0/2,1/1,1/1} &
\drawPoly{\eps}{0/2,1/2} &
\drawPoly{\eps}{0/3,1/1} &
\drawPoly{\eps}{0/3,2/1} &
\drawPoly{\eps}{0/4}
\\
\perm{4,1,2,3} &
\perm{4,1,3,2} &
\perm{4,2,1,3} &
\perm{4,3,1,2} &
\perm{4,3,2,1}
\end{tabular}
\caption{DCC polyominoes of area 4 and their corresponding $(231,3124)$-avoiding permutations of size 4.}
\label{fig:4dccpPerm}
\end{figure}

\begin{remark}
In the above bijection, the elements of $\Av_n^{k\mapsto1}(231,3124)$ correspond to DCC polyominoes of area $n$ whose first column has exactly $k$ cells.
\end{remark}

\begin{remark}
Another class of combinatorial objects counted by the odd-indexed Fibonacci numbers are the nondecreasing Dyck paths. These are Dyck paths for which the sequence of the altitudes of the valleys is nondecreasing. A bijection between DCC polyominoes and nondecreasing Dyck paths can be found in Deutsch \& Prodinger \cite[Section 3]{DP03}. 

Moreover, as discussed in Vella~\cite{Vella03}, the so-called standard bijection\footnote{Also known as first-return bijection.} between $132$-avoiding permutations and Dyck paths gives a bijection between $(132,3241)$-avoiding permutations and nondecreasing Dyck paths.
\end{remark}

%%%%%%%%%%%%%%%%%%%%%%%%%%%%%%%%%%%%%%%%%%%%


\begin{thebibliography}{99}
%
\bibitem{Atk99} M.~D.~Atkinson, Restricted permutations, {\em Discrete Math.} \textbf{195} (1999), no.~1--3, 27--38.
%
\bibitem{BPS93} E.~Barcucci, R.~Pinzani, and R.~Sprugnoli, Directed column-convex polyominoes by recurrence relations, {\em Lecture Notes in Computer Science}, No. 668, Springer, Berlin (1993), pp. 282--298.
%
\bibitem{DD92} M.~Delest and S.~Delucq, Enumeration of directed column-convex animals with a given perimeter and area, {\em Croatica Chemica Acta} \textbf{66} (1993), no.~1, 59--80.
%
\bibitem{DP03} E.~Deutsch and H.~Prodinger, A bijection between directed column-convex polyominoes and ordered trees of height at most three, {\em Theoretical Comp. Science} \textbf{307} (2003) 319--325.
%
\bibitem{Kitaev11} S.~Kitaev, \emph{Patterns in permutations and words}, Monographs in Theoretical Computer Science, an EATCS Series, Springer, Heidelberg, 2011.
%
\bibitem{Kratt01} C.~Krattenthaler, Permutations with restricted patterns and Dyck paths, {\em Adv. Appl. Math.} \textbf{27} (2001), 510--530.
%
\bibitem{Mar13} E.~Marberg, Crossings and nestings in colored set partitions, {\em Electron. J. Combin.} \textbf{20} (2013), no.~4, \#P6.
%
\bibitem{Tenner07} B.~E.~Tenner, Pattern avoidance and the Bruhat order, {\em J. Combin. Theory Ser. A} \textbf{114} (2007) 888--905.
%
\bibitem{oeis} OEIS Foundation Inc. (2025), The On-Line Encyclopedia of Integer Sequences, Published electronically at \url{https://oeis.org/}.
%
\bibitem{Vella03} A.~Vella, Pattern avoidance in cyclically ordered structures, {\em Electron. J. Combin.} \textbf{9}
(2003), \#R18.
%
\bibitem{West96} J.~West, Generating trees and forbidden subsequences, {\em Discrete Math.} \textbf{157} (1996), no.~1--3, 363--374.
%
\bibitem{Wgfology} H.~Wilf, \emph{Generatingfuntionology}, 3rd ed., A K Peters, Ltd., Wellesley, MA, 2006.
\end{thebibliography}
\end{document}